\documentclass[11pt]{amsart}

\usepackage{amsmath,amssymb,cite}
\usepackage{amsthm}
\usepackage{amsrefs}
\usepackage{mathtools}
\usepackage{tikz}
\usetikzlibrary{patterns}
\usepackage{diagbox}
\usepackage{array}
\usepackage{graphicx}
\usepackage{relsize}
\usepackage{enumitem}
\usepackage[titletoc]{appendix}

\usepackage{caption}
\usepackage{subcaption}  
\usepackage{multirow}
\usepackage{tabularx}
\usepackage{tikz-cd}
\usepackage[plainpages=false,pdfborder={0 0 0}]{hyperref}

\usepackage[a4paper,top=3cm,bottom=2cm,left=3cm,right=3cm,marginparwidth=1.75cm]{geometry}

\newtheorem{thm}{Theorem}[section]

\newtheorem{prop}[thm]{Proposition}
\newtheorem{cor}[thm]{Corollary}
\theoremstyle{definition}
\newtheorem{defn}[thm]{Definition}
\newtheorem{example}[thm]{Example}

\theoremstyle{remark}
\newtheorem{rmk}[thm]{Remark}

\newcommand{\R}{\mathbb{R}}

\numberwithin{equation}{section}

\newcommand{\pa}{\partial}

\newcommand{\D}{\nabla}
\newcommand{\hess}{\operatorname{Hess}}

\newcommand{\Ric}{\mathrm{Ric}}

\begin{document}

\title[Obata-Type Rigidity]{Obata-Type Rigidity on Static Manifolds with Boundary}


\author{Hongyi Sheng}
\address{Institute for Theoretical Sciences, Westlake Institute for Advanced Study, Westlake University, Hangzhou, Zhejiang Province 310024, China}
\email{shenghongyi@westlake.edu.cn}
\thanks{}

\author{Kai-Wei Zhao}
\address{Department of Mathematics, University of California, Irvine, Irvine, CA, 92697, USA}
\email{kaiweiz2@uci.edu}
\thanks{}


\date{}

\dedicatory{}

\begin{abstract}
We investigate static metrics on simple manifolds with compact boundary and establish an Obata-type rigidity theorem. We identify new sufficient geometric conditions under which the combined curvature map $g\mapsto (R_g, H_g)$ is a local surjection. Consequently, we demonstrate that in contrast to manifolds without boundary, where staticity obstructs deformability, the scalar curvature map can be locally surjective at static metrics on manifolds with boundary. 
\end{abstract}

\maketitle


\section{Introduction}
The deformation theory of scalar curvature is a central topic in geometric analysis and general relativity, particularly in understanding the relationship between the scalar curvature map and the presence of \emph{static} metrics. Fischer and Marsden \cite{F-M} showed that on a closed manifold, the scalar curvature map is generically a local surjection. Later, Corvino \cite{C} clarified the relationship with static metrics, proving that the scalar curvature map is locally surjective at a metric $g$ if and only if $g$ is not static. In this paper, we investigate how the presence of a boundary fundamentally alters this landscape.

We begin by establishing the precise notion of staticity in the boundary setting. While various definitions of ``static manifold with boundary" appear in the literature, we adopt the following formulation similar to that of \cite{Sheng2}, which couples the interior static equation with a specific boundary condition related to the linearization of the boundary mean curvature.

\begin{defn}
Let $(M,g)$ be a complete, connected, smooth manifold with boundary. A non-trivial function $V\in C^\infty(M)$ is called a \emph{static potential} if it satisfies the equation
\begin{equation}\label{static int}
    -\left(\Delta V\right) g+\operatorname{Hess}V - V \operatorname{Ric}  = 0\quad\text{ in } M.
\end{equation}
A metric $g$ admitting such a $V$ is said to be \emph{static}. Furthermore, we say that a static potential $V$ is \emph{admissible} if it satisfies the boundary condition 
\begin{equation}\label{static bdry}
    V_\nu \hat{g} = V h \quad \text{ on } \partial M,
\end{equation}
where $V_\nu$ denotes the normal derivative with respect to the outward unit normal $\nu$, $\hat{g}$ is the induced metric on $\partial M$, and $h$ is the second fundamental form defined by $h(X,Y) = -\langle \nu, D_{X}Y \rangle$ for all vector fields $X,Y$ tangent to $\partial M$. The pair $(M,g)$ is called a \emph{static manifold with boundary} if it admits an admissible static potential.
\end{defn}

Static manifolds with boundary exhibit strong intrinsic geometric constraints. As shown in \cites{Sheng, Sheng2}, the existence of an admissible static potential imposes significant rigidity on the curvatures:

\begin{thm}[\cite{Sheng}*{Theorem 3.3}, \cite{Sheng2}*{Remark 2.2}]\label{basic}
If $(M,g)$ is a static manifold with boundary, then the scalar curvature of $M$ is constant, the boundary $\partial M$ is umbilic, and the mean curvature is locally constant on $\partial M$. Moreover, we have $\operatorname{Ric}_{i\nu} = 0 \,\,\, (i = 1, \ldots, n-1)$ on $\partial M$.
\end{thm}

To analyze the deformation theory of these manifolds, we follow \cite{Sheng} and define the operators $L^*_g: C^{\infty}(\overline{M}) \longrightarrow \mathcal{C}^{\infty}(\overline{M})$ by
$$
L^*_g u = -\left(\Delta u\right) g+\operatorname{Hess}u - u \operatorname{Ric},
$$
and $\Phi^*_g: C^{\infty}(\overline{M}) \longrightarrow \mathcal{C}^{\infty}(\overline{M}) \times \mathcal{C}^{\infty}(\partial M)$ by
$$
\Phi^*_g u = (L^*_g u, u_\nu \hat{g}-u h).
$$
Recall that the map $g\mapsto (R_g, H_g)$ is a local surjection if and only if $\operatorname{ker} \Phi^*_g = \{0\}$ \cite{Sheng}. 
A primary focus of this work is the regime where $\operatorname{ker} \Phi^*_g = \{0\}$ despite the fact that $\operatorname{ker} L^*_g \neq \{0\}$. In such a setting, even though $g$ is static (admitting a static potential in the kernel of $L^*_g$), the combined curvature map $g\mapsto (R_g, H_g)$ remains a local surjection. This highlights a key distinction from the boundaryless case: for manifolds with boundary, staticity does not necessarily obstruct deformability.

To state our main rigidity and deformation results, we restrict our attention to the connected case.

\begin{defn}
    A manifold with boundary is called \emph{simple} if $M$ and its boundary $\partial M$ are both connected.
\end{defn}

Our first main result is a classification theorem. Let $(M,g)$ be a simple static manifold with compact boundary admitting a positive static potential $V$. Under specific curvature conditions, we show that $V$ satisfies an Obata-type equation with Robin boundary conditions. While Obata's equation has been studied extensively (see, e.g., \cites{K, T}), the inclusion of the boundary term leads to the following new rigidity classification.

\begin{thm}[Obata-Type Rigidity] 
Let $(M,g)$ be a simple static manifold with compact boundary admitting a positive static potential $V$. If $(M,g)$ is non-compact, we additionally assume the decay condition \eqref{decay}. Suppose that the scalar curvature $R$ and mean curvature $H$ satisfy 
\begin{equation*}
 -\frac{n}{n-1}H^2\le R\le -H^2.
\end{equation*} 
Then $V$ satisfies an Obata-type equation subject to Robin boundary conditions. Moreover, $(M,g)$ is isometric to one of the following: 
\begin{itemize} 
\item[I.] A Ricci-flat manifold with totally geodesic boundary; 
\item[II.] The non-compact warped product $M \cong (-\infty,0] \times \Sigma$ equipped with the metric $g = dt^2 + e^{\frac{2Ht}{n-1}} g_{\Sigma}$, where $(\Sigma, g_{\Sigma})$ is a compact, Ricci-flat manifold of dimension $n-1$. \end{itemize}
\end{thm}
We say that $(M,g)$ is of \emph{Obata type} if it is isometric to one of the cases described above.

\bigskip
Building on this classification, we establish sufficient geometric conditions for the deformability of scalar curvature, which require rough information about static potentials near infinity. This constitutes our second main result.

\begin{thm}\label{surj}
Let $(M,g)$ be a complete, connected, smooth manifold with boundary, and suppose that $g$ is static with a positive static potential. The map $g\mapsto (R_g, H_g)$ is a local surjection if any of the following conditions holds:
\begin{itemize}
    \item[1.] The mean curvature $H$ is not locally constant;
    \item[2.] The boundary $\partial M$ is not umbilical;
    \item[3.] $(M,g)$ is a simple manifold with compact boundary that is not of Obata type, satisfies the decay condition \eqref{decay} if non-compact, and has scalar curvature satisfying
    \begin{equation*}
        -\frac{n}{n-1}H^2\le R\le -H^2.
    \end{equation*}
\end{itemize}
\end{thm}

\begin{rmk}
    Corvino \cite{C} proved that for manifolds without boundary, the scalar curvature map is a local surjection at $g$ if and only if $g$ is not static. As a corollary of Theorem~\ref{surj}, we demonstrate that for manifolds with boundary, the scalar curvature map can be a local surjection even if $g$ is static, provided that the geometric conditions in Theorem~\ref{surj} are met.
\end{rmk}

The remainder of this paper is structured as follows. In Section~\ref{sec2}, we establish a key inequality for simple static manifolds with compact boundary which leads to the Obata-type rigidity theorems. As an application, we prove Theorem~\ref{surj}. In Section~\ref{sec3}, we present examples of static manifolds in general relativity that are not of Obata type and admit a locally surjective curvature map $g\mapsto (R_g, H_g)$.

\section*{Acknowledgments}
We are grateful to Lan-Hsuan Huang for helpful discussions. This work was initiated during the authors' visit to the Simons Laufer Mathematical Sciences Institute (formerly MSRI) in Berkeley, California, during the Fall 2024 semester. We thank the institute for its hospitality and support.

\section{Obata-type rigidity theorem}\label{sec2}
We begin this section with some preliminary observations.

By taking the trace of equations \eqref{static int} and  \eqref{static bdry}, we obtain
\begin{align}\label{static trace}
\left\{\begin{aligned}
\Delta V+\frac{R}{n-1} V & =0 \quad & \text {in }  &M,\\
V_\nu-\frac{H}{n-1} V & =0 & \text {on }  &\partial M.
\end{aligned}\right.
\end{align}
Consequently, the system \eqref{static int}–\eqref{static bdry} is equivalent to
\begin{align}\label{static eqn}
\left\{\begin{aligned} 
\operatorname{Hess}V & = \left(\operatorname{Ric} - \frac{R}{n-1}g\right) V \qquad & \text{in } &M,\\ 
V_\nu & =\frac{H}{n-1} V & \text{on } &\partial M.
\end{aligned}\right.
\end{align}

Recall that on a simple static manifold with compact boundary, the mean curvature $H$ of $\partial M$ is constant. Hence, we may extend $H$ as a constant function on the whole of $M$.

In what follows, we derive a useful inequality valid for such manifolds.

\begin{prop}\label{integral eqn}
Let $(M,g)$ be a compact simple static manifold with boundary, admitting a static potential $V>0$. Then 
\begin{equation*}
    \int_M V^{-1}\left|\operatorname{Hess}V +\frac{R+H^2}{n-1}Vg\right|^2 d\mu_g = \frac{R+H^2}{n-1}\left(R+\frac{n}{n-1}H^2\right)\int_M V d\mu_g.
\end{equation*}
\end{prop}
\begin{proof}
Define a symmetric $(0,2)$-tensor $S$ on $M$ by
$$
S \triangleq\operatorname{Ric} + \frac{H^2}{n-1}g.
$$  
From \eqref{static eqn}, we have
$$
VS = \operatorname{Hess}V +\frac{R+H^2}{n-1}Vg.
$$
Using the second Bianchi identity and the property that the scalar curvature $R$ is constant (Theorem~\ref{basic}), we find that $S$ is divergence-free:
$$
\operatorname{div}S = \operatorname{div}\operatorname{Ric} = \frac12 DR = 0.
$$
So we have
$$
\operatorname{div}\left(S(DV, \cdot)\right) = \operatorname{div}S\cdot DV + \left<\operatorname{Hess}V, S\right> = \left<\operatorname{Hess}V, S\right>.
$$
Thus,
\begin{align}
V^{-1}\left|\operatorname{Hess}V +\frac{R+H^2}{n-1}Vg\right|^2 & = \left<VS, S\right>\notag\\
    & = \left<\operatorname{Hess}V, S\right> + \frac{R+H^2}{n-1}V\left<g, S\right>\notag\\
    & = \operatorname{div}\left(S(DV, \cdot)\right) + \frac{R+H^2}{n-1}\left(R+\frac{nH^2}{n-1}\right)V\label{eq: local eqn}.
\end{align}

Integrating this expression over $M$ and applying the Divergence Theorem yields:
$$
\int_M V^{-1}\left|\operatorname{Hess}V +\frac{R+H^2}{n-1}Vg\right|^2 d\mu_g = \int_{\partial M} S(DV, \nu)d\sigma_g + \frac{R+H^2}{n-1}\left(R+\frac{nH^2}{n-1}\right)\int_M Vd\mu_g.
$$

To complete the proof, we must show that the boundary integral vanishes:
$$
\int_{\partial M} S(DV, \nu)d\sigma_g = 0.
$$
We choose a local orthonormal frame $\{e_1, \dots, e_{n-1}, \nu\}$ on $\partial M$, where $\nu$ is the outward unit normal vector. On $\partial M$, it is known that $\operatorname{Ric}_{i\nu} = 0$ for $i = 1, \dots, n-1$. Thus, $S(DV, \nu)$ becomes:
$$
\begin{aligned}
    S\left(DV, \nu\right) & = \left(\operatorname{Ric} + \frac{H^2}{n-1}g\right)\left(\sum_{i=1}^{n-1} V_ie_i+V_\nu \nu, \nu\right)\\
    & = \left(R_{\nu\nu} + \frac{H^2}{n-1}\right)V_\nu\\
    & = \left(R_{\nu\nu} + \frac{H^2}{n-1}\right)\frac{H}{n-1}V.
\end{aligned}
$$
On the other hand, from \eqref{static trace}, \eqref{static eqn}, and the Gauss equation, we have the following on $\partial M$:
$$
\begin{aligned}
0 &=\Delta V+\frac{R}{n-1} V\\
    &=\Delta_{\partial M} V+HV_\nu+V_{\nu\nu}+\frac{R}{n-1} V\\
	&=\Delta_{\partial M} V+\left(\frac{H^2}{n-1} +R_{\nu\nu}\right) V.\\
\end{aligned}
$$
Integrating this final equation over the boundary $\partial M$, we get 
$$
0 = \int_{\partial M} \left(\frac{H^2}{n-1} +R_{\nu\nu}\right) Vd\sigma_g.
$$
As a result, 
$$
\int_{\partial M} S(DV, \nu)d\sigma_g = \frac{H}{n-1}\int_{\partial M} \left(R_{\nu\nu} + \frac{H^2}{n-1}\right)Vd\sigma_g = 0.
$$
\end{proof}

Subject to appropriate decay conditions, the previous result extends to the non-compact setting.
\begin{prop}[Extension to Non-Compact Manifolds]
Let $(M,g)$ be a non-compact, simple static manifold with compact boundary, admitting a static potential $V>0$. Suppose that the following decay condition holds along a compact exhaustion $\{B_r\}$ of $M$:
\begin{equation}\label{decay}
    \liminf_{r\rightarrow\infty} \int_{\partial B_r} S(DV, \nu)d\sigma_g = \liminf_{r\rightarrow\infty} \int_{\partial B_r} \left(\operatorname{Ric} + \frac{H^2}{n-1}g\right)(DV, \nu)d\sigma_g \le 0.
\end{equation}
Then the following inequality holds:
\begin{equation}\label{identity}
    \int_M V^{-1}\left|\operatorname{Hess}V +\frac{R+H^2}{n-1}Vg\right|^2 d\mu_g \le \frac{R+H^2}{n-1}\left(R+\frac{n}{n-1}H^2\right)\int_M V d\mu_g.
\end{equation}
This inequality extends to the case where the integrals of the non-negative functions on both sides diverge to $+\infty$.
\end{prop}

\begin{proof}
We begin with the local identity \eqref{eq: local eqn}. Integrating over $B_r \, (\text{for }r\gg 1)$ and applying the Divergence Theorem yields:
\begin{align*}
    \int_{B_r} V^{-1}\left|\operatorname{Hess}V +\frac{R+H^2}{n-1}Vg\right|^2 d\mu_g & = \int_{\partial M} S(DV, \nu)d\sigma_g + \int_{\partial B_r} S(DV, \nu)d\sigma_g\\
    & \quad + \frac{R+H^2}{n-1}\left(R+\frac{nH^2}{n-1}\right)\int_{B_r} Vd\mu_g\\
    & = \int_{\partial B_r} S(DV, \nu)d\sigma_g + \frac{R+H^2}{n-1}\left(R+\frac{nH^2}{n-1}\right)\int_{B_r} Vd\mu_g.
\end{align*}
Taking the limit inferior as $r\rightarrow\infty$ and using the decay condition \eqref{decay}, by the monotone convergence theorem, we find:
\begin{align*}
   & \int_M V^{-1}\left|\operatorname{Hess}V +\frac{R+H^2}{n-1}Vg\right|^2 d\mu_g\\
   = & \liminf_{r\rightarrow\infty}\int_{\partial B_r} S(DV, \nu)d\sigma_g + \frac{R+H^2}{n-1}\left(R+\frac{nH^2}{n-1}\right)\int_{M} Vd\mu_g\\
   \le & \frac{R+H^2}{n-1}\left(R+\frac{n}{n-1}H^2\right)\int_M V d\mu_g.
\end{align*}
\end{proof}

Observing that the left-hand side of \eqref{identity} is non-negative, the inequality implies that if the right-hand side is non-positive, both sides must vanish. This yields the following rigidity result.

\begin{cor}\label{cor: obata type}
Let $(M,g)$ be a simple static manifold with compact boundary admitting a positive static potential $V$. If $(M,g)$ is non-compact, we additionally assume the decay condition \eqref{decay}. Suppose that
    \begin{equation}\label{eq: curv.cond}
        (R+H^2)\left(R+\frac{n}{n-1}H^2\right) \le 0.
    \end{equation}
    Then $V$ satisfies the following Obata-type equation ($VS = 0$) subject to Robin boundary conditions:
    \begin{align}\label{Obata Robin}
        \left\{\begin{aligned} 
        \operatorname{Hess}V +\frac{R+H^2}{n-1}Vg &=0 \qquad & \text{in } & M,\\ 
        V_\nu - \frac{H}{n-1} V & = 0 & \text{on } &\partial M.
        \end{aligned}\right.
    \end{align}
    Consequently, $(M, g)$ is Einstein. 
\end{cor}

Under these hypotheses, we necessarily have
\begin{equation}\label{eq: zero.cond}
        (R+H^2)\left(R+\frac{n}{n-1}H^2\right) = 0.
    \end{equation}
We examine the two possibilities arising from \eqref{eq: zero.cond} separately.

\begin{thm}\label{thm:Obata1}
Suppose that $R = -H^2$. Then $(M,g)$ is a Ricci-flat manifold with totally geodesic boundary, and $V$ is a positive constant.
\end{thm}
\begin{proof}
    If $R = -H^2$, the interior equation in \eqref{Obata Robin} reduces to $\operatorname{Hess}V =0$. So by \eqref{static int}, we obtain $\operatorname{Ric} = 0$, which further implies $R = H = 0$. Thus, the system \eqref{Obata Robin} becomes:
\begin{align}\label{Obata Robin1}
\left\{\begin{aligned} 
\operatorname{Hess}V &=0 \qquad & \text{in } & M,\\ 
V_\nu & = 0 & \text{on } &\partial M.
\end{aligned}\right.
\end{align}
It follows that the vector field $DV$ is parallel. Moreover, the boundary condition shows that $DV$ is everywhere orthogonal to the normal vector $\nu$; consequently, $DV$ is tangent to the boundary $\partial M$.

Since $DV$ is parallel, its integral curves $\gamma(t)$ are geodesics defined for all $t\in\mathbb{R}$. Along these geodesics, the function $V$ is linear:
$$
V(\gamma(t))=ct+V(\gamma(0)),
$$
where $c=|DV|^2$ is a constant.

However, since $V>0$ everywhere, the slope $c$ must vanish. This implies $|DV|^2 = 0$, so $V$ must be a positive constant. Consequently, this case yields no additional geometric information about $M$.
\end{proof}

We now turn to the second case, where $R = -\frac{n}{n-1}H^2 < 0$. For convenience, define $\lambda \triangleq \tfrac{H}{n-1} \neq 0$. In this setting, the system \eqref{Obata Robin} reduces to:
\begin{align}\label{Obata Robin2}
\left\{\begin{aligned} 
\operatorname{Hess}V &= \lambda^2Vg  \qquad & \text{in } & M,\\ 
V_\nu &= \lambda V  & \text{on } &\partial M.
\end{aligned}\right.
\end{align}
Furthermore, $\operatorname{Ric} = - (n-1) \lambda^2g$. 

\begin{thm}\label{thm:Obata2}
Suppose that $R = -\frac{n}{n-1}H^2 < 0$. Then $M$ is non-compact and is isometric to the warped product
$$M \cong (-\infty,0] \times \Sigma$$
equipped with the metric
$$g = dt^2 + e^{2\lambda t} g_{\Sigma},$$
where:
\begin{itemize}
    \item the boundary $\partial M$ is identified with the slice $\{0\} \times \Sigma$;
    \item $(\Sigma, g_{\Sigma})$ is a compact, Ricci-flat manifold of dimension $n-1$;
    \item the static potential is given by $V(t, x) = V_0 e^{\lambda t}$ for some positive constant $V_0$.
\end{itemize}
Furthermore, if $(M,g)$ has constant sectional curvature, then $\Sigma$ must be flat. In this case, $M$ is isometric to a hyperbolic cusp.
\end{thm}

\begin{proof}
Define the function $Q \triangleq |D V|^2 - \lambda^2 V^2$.
Recall that for any vector field $X\in TM$,
\begin{equation*}
    D_X \lvert D V\rvert^2 = 2g( D_X DV, DV) = 2\hess V(D V, X).
\end{equation*}
Using this identity together with the interior equation from \eqref{Obata Robin2}, we compute the derivative of $Q$:
\begin{align*}
    D_X Q & = 2 \hess V(D V, X) - 2 \lambda^2 V D_X V\\
    & = 2 \lambda^2 V g(D V, X) - 2 \lambda^2 V g(D V, X) = 0.
\end{align*}
This implies that $Q\equiv C$ is constant on $M$. Evaluating $Q$ on the boundary $\partial M$ and using the boundary condition $V_\nu = \lambda V$, we have:
\begin{equation}
    C = \lvert D V\rvert^2 - \lambda^2 V^2= \lvert D V\rvert^2 - V_\nu^2 = \lvert \D V\rvert^2 \geq 0.
\end{equation}

We claim that $C = 0$. Suppose for the sake of contradiction that $C>0$. From the equation above, $\lvert \D V\rvert^2 = C > 0$, so $\D V$ is nowhere vanishing on $\partial M$. Since $\partial M$ is complete, for any point $p\in\partial M$, there exists an integral curve $\gamma(s): \R \rightarrow \partial M$ of the vector field $-\D V$ starting at $p = \gamma(0)$. Along this curve, we have: 
\begin{align*}
    V(\gamma(s)) = V(\gamma(0)) + \int_0^s \frac{d}{d\tau} V(\gamma(\tau))\, d\tau
    = V(p) - \int_0^s \lvert \D V(\gamma(\tau))\rvert^2 d\tau
    = V(p) - s C.
\end{align*}
Since $C>0$, for sufficiently large $s$, $V(\gamma(s))$ becomes negative. This contradicts the assumption that $V>0$. Thus, we must have $C=0$.

Consequently, $\D V = 0$, which implies that $V \equiv V_0> 0$ is constant on the boundary $\partial M$. Furthermore, in the interior $M$, the condition $Q=0$ implies:
\begin{equation}\label{eq: DV}
    \lvert D V\rvert = \lvert \lambda \rvert V.
\end{equation}
Define the function $f \triangleq \log V$. Then \eqref{eq: DV} implies $\lvert D f\rvert = \lvert \lambda\rvert$. Since $f$ has constant gradient norm, its integral curves are geodesics. Specifically, since $V_\nu = \lambda V > 0$ (assuming $\lambda>0$ without loss of generality) on the boundary, $Df$ points outward. Let $t$ be the signed distance function from $\pa M$, oriented such that $Dt$ coincides with the outward normal $\nu$ at $\partial M$. Since $\lvert D f\rvert$ is constant and $Df$ is proportional to $\nu$ at the boundary, the level sets of $f$ coincide with the level sets of $t$. Thus, $M$ splits topologically as $(-\infty, 0] \times \Sigma$ (where $\Sigma \cong \partial M$), and the metric splits as
$$g = dt^2 + g_t,$$
where $g_t$ is the induced metric on the slice $\{t\}\times \Sigma$. The relation $\lvert D f\rvert = \lvert \lambda\rvert$ implies $\partial_t f = \lambda$, so $f(t) = \log V_0 + \lambda t$. Consequently, 
$$V(t) = V_0 e^{\lambda t}.$$

We now determine the metric $g_t$. Using the Hessian condition in \eqref{Obata Robin2} and the fact that $Dt = \partial_t$, the evolution equation for the metric is:
\begin{align*}
    \frac{\partial}{\partial t} g_t(X,Y) & = \mathcal{L}_{D_t}g (X,Y) = 2\hess t(X,Y)\\
    & = 2 g\left(D_X D t, Y\right)\\
    & = 2g\left(D_X \left(\frac{D V}{\lambda V}\right), Y\right)\\
    & = \frac{2}{\lambda V}g\left(D_X DV, Y\right)\\
    & = \frac{2}{\lambda V}\hess V(X,Y)\\
    & =2\lambda g(X,Y) = 2\lambda g_t(X,Y),
\end{align*}
where $X,Y$ are tangent to $\Sigma$. Integrating this ODE yields $g_t = e^{2\lambda t} g_\Sigma$. Thus, the full metric is $g = dt^2 + e^{2\lambda t}g_\Sigma$.

We check the compactness of $M$. The coordinate $t$ ranges in $(-\infty,0]$. Since $V(t) = V_0 e^{\lambda t}$ is strictly positive and non-critical for all $t$, the flow lines cannot terminate at an interior critical point, nor can there be another boundary component (where the boundary condition would fail). Thus, $M$ is non-compact.

To characterize $\Sigma$, we compute the Ricci curvature. By \cite{O}*{Corollary 7.43}, for the warped product metric with warping function $\rho = e^{\lambda t}$, the Ricci curvature restricted to horizontal vectors $X,Y\in T\Sigma$ is:
\begin{align*}
    \operatorname{Ric}_M(X,Y) & = \operatorname{Ric}_{\Sigma}(X,Y) - \left(\frac{\Delta_\R \rho}{\rho}  + (n-2) \frac{\lvert D \rho\rvert^2}{\rho^2}\right) g(X,Y)\\
    & = \operatorname{Ric}_{\Sigma}(X,Y)- (n-1)\lambda^2 g(X,Y),
\end{align*}
where $\Delta_\R = \pa_t^2$ is the Laplacian on the base $\R$. Comparing this with the Einstein condition $\operatorname{Ric}_M = -(n-1)\lambda^2 g$, we conclude that $\operatorname{Ric}_{\Sigma} = 0$.
Thus, $(\Sigma, g_{\Sigma})$ is  Ricci-flat.

Finally, assume that $(M, g)$ has constant sectional curvature $K$. Given the Ricci curvature derived above, we must have $K = -\lambda^2$. Consequently, the Riemann curvature tensor is given by
$$\operatorname{Rm}_M(X,Y)Z = -\lambda^2\left(g(Z,X) Y - g(Z, Y) X\right).$$
On the other hand, by \cite{O}*{Proposition 7.42}, the curvature tensor of the warped product restricted to vectors $X,Y,Z$ tangent to $\Sigma$ satisfies: 
\begin{align*}
    \operatorname{Rm}_M(X,Y)Z &= \operatorname{Rm}_\Sigma(X,Y)Z - \frac{\lvert D\rho\rvert^2}{\rho^2} (g(Z,X) Y - g(Z, Y) X) \\
    &= \operatorname{Rm}_\Sigma(X,Y)Z  -\lambda^2(g(Z,X) Y - g(Z, Y) X).
\end{align*}
Comparing these two expressions yields $\operatorname{Rm}_\Sigma \equiv 0$, implying that $\Sigma$ is flat. Since $\Sigma$ is compact, it follows that $\Sigma \cong \mathbb{R}^{n-1}/\Gamma$, and consequently, $M$ is isometric to a hyperbolic cusp.
\end{proof}

We combine the above results to prove Theorem~\ref{surj}.

\begin{proof}[Proof of Theorem~\ref{surj}]
We focus on the case where Condition 3 holds. Proceeding by contradiction, suppose that the map $g\mapsto (R_g, H_g)$ is not a local surjection. It follows that $(M,g)$ must be a simple static manifold with compact boundary. By Condition 3, the curvatures satisfy the inequality
\begin{equation*}
    \left(R+H^2\right)\left(R+\frac{n}{n-1}H^2\right)\le 0.
\end{equation*}
Consequently, Corollary~\ref{cor: obata type} implies that $(M,g)$ must be of Obata type. However, this contradicts the hypothesis in Condition 3 that $(M,g)$ is not of Obata type. Therefore, the map must be a local surjection.
\end{proof}


\section{Examples}\label{sec3}
In this section, we present examples of simple manifolds with compact boundary (subject to the decay condition \eqref{decay} if non-compact) equipped with a static metric $g$ and a positive static potential. These examples are not of Obata type and have scalar curvature satisfying 
\begin{equation*}
-\frac{n}{n-1}H^2\le R\le -H^2. \end{equation*} 
Consequently, by Theorem~\ref{surj}, the map $g\mapsto (R_g, H_g)$ is a local surjection.

\begin{example}[Schwarzschild metric]
Let $m>0$. Consider the Schwarzschild metric $g = \phi^{\frac{4}{n-2}} \delta$ on $\mathbb{R}^{n} \setminus \{\mathbf{0}\}$, where $\delta$ is the Euclidean metric, $r = |\vec{x}|$, and the conformal factor is defined by $\phi = 1 + \frac{m}{2r^{n-2}}$. Let $M$ be the region outside the horizon ($r^{n-2} > \frac{m}{2}$), so that the boundary $\partial M$ corresponds to the horizon itself. Then $(M, g)$ is a simple manifold with compact boundary.

\textbf{Geometric Quantities.}
As shown in \cite{Sheng2}, $g$ is static with a potential $V$ given by:
$$V = \frac{1 - \frac{m}{2r^{n-2}}}{1 + \frac{m}{2r^{n-2}}} = \frac{2}{\phi} - 1.$$
Note that $V > 0$ in the interior of $M$ and vanishes on the horizon $\partial M$.

Recall that $(M, g)$ is scalar-flat ($R = 0$), and since the horizon $\partial M$ is minimal, the mean curvature $H = 0$. Consequently, $R = -H^2 = 0$, and only Type I rigidity is possible.

\textbf{Decay Condition.}
We verify that $\operatorname{Ric}(DV, \nu)$ satisfies the decay condition at infinity. Let $\vec{e} = \frac{\vec{x}}{r}$ be the Euclidean unit radial vector. The $g$-unit outward normal to the coordinate sphere $\Sigma_r$ is:
    $$\nu = \phi^{-\frac{2}{n-2}} \vec{e}.$$
The Euclidean gradient of the potential, $DV$, is calculated as:
    $$DV = \frac{m(n-2)}{r^{n-1} \phi^2} \vec{e} = \frac{m(n-2)}{r^{n-1}} \phi^{\frac{2}{n-2}-2} \nu.$$
Using the standard formula for the Ricci tensor of a conformally flat metric, we observe the leading-order decay:
    $$\operatorname{Ric}_{ij} \sim \frac{m(n-2)}{r^n}\left(\delta_{ij} - n\frac{x_ix_j}{r^2}\right).$$
As a result,    
\begin{align*}
        \operatorname{Ric}(DV, \nu) = \frac{m(n-2)}{r^{n-1}} \phi^{\frac{2}{n-2}-2} \operatorname{Ric}(\nu, \nu) \sim \frac{-m^2(n-1)(n-2)^2}{r^{2n-1}}.
    \end{align*}
 Integrating over the sphere $\Sigma_r$ (where $\operatorname{Area}_g(\Sigma_r) \sim \omega_{n-1} r^{n-1}$), the limit becomes:
\begin{align*}
\liminf_{r\to\infty} \int_{\Sigma_r} S(DV, \nu)  d\sigma_g = \lim_{r\to\infty} \int_{\Sigma_r} \operatorname{Ric}(DV, \nu) d\sigma_g = \lim_{r\to\infty} \left( \frac{C}{r^{2n-1}} \cdot r^{n-1} \right) = \lim_{r\to\infty} \frac{C}{r^n} = 0.
\end{align*}

Finally, since $(M, g)$ is not Ricci-flat, it cannot be of Type I and hence is not of Obata type.
\end{example}

\begin{example}[Generalized Kottler metric]
Let $m\ge0$. Let $(\Sigma, b)$ be a closed $(n-1)$-dimensional Riemannian manifold with constant Ricci curvature $\operatorname{Ric}_b = (n-2)b$.
Define the metric $g$ on $M = [r_c, \infty) \times \Sigma$ by
\begin{equation*}
    g = \frac{1}{r^2 + 1 - 2m r^{2-n}} dr^2 + r^2 b,
\end{equation*} 
where $r_c = (2m)^{\frac{1}{n-2}}$. We first verify that $(M, g)$ is well-defined by showing $r_c> r_0$, where $r_0$ is the largest root of $u(r) = r^2 + 1 - 2m r^{2-n}$.

\textbf{Case I} ($m=0$): $u(r) = r^2+1>0$ for all $r$, so the metric is well-defined for all $r$. 

\textbf{Case II} ($m>0$): Since $u'(r) = 2r(1+m(n-2)r^{-n})>0$ for $r>0$, $u$ is strictly increasing. Since
$$
u(r_c) = (2m)^{\frac{2}{n-2}} + 1 - 2m\cdot(2m)^{-1} = (2m)^{\frac{2}{n-2}}>0 = u(r_0),
$$
monotonicity implies $r_c > r_0$.

Thus, $(M, g)$ is a smooth simple manifold with a compact boundary at $r = r_c$.

\bigskip
\textbf{Geometric Quantities.} As shown in \cite{H-J}, $g$ is static with a static potential $V = \sqrt{r^2 + 1 - 2m r^{2-n}} >0$. Then we may rewrite the metric as 
$$g = \frac{1}{V^2} dr^2 + r^2 b.$$
The Ricci curvature components are given by
\begin{align*}
    \Ric_{rr} &= -\frac{n-1}{r}\cdot\frac{\left(V^2\right)'}{2}g_{rr} = -(n-1) \left(1 + m(n-2)r^{-n}\right)g_{rr},\\
    \Ric_{ij} &= \left(\frac{(n-2)-(n-2)V^2}{r^2} - \frac{\left(V^2\right)'}{2r}\right)g_{ij} = -\left((n-1) - m(n-2)r^{-n}\right)g_{ij}.
\end{align*}
It follows that the scalar curvature is constant: $R = -n(n-1).$

Let $\nu = V \partial_r$ be the outward unit normal to the $r$-level sets $\Sigma_r$. The mean curvature $H$ of $\Sigma$ is:
$$
H = \operatorname{div}(-\nu) = -\frac{n-1}{r_c}V(r_c) = -(n-1).
$$
So $R = -\frac{n}{n-1}H^2 < 0$, and only Type II rigidity is possible.

\textbf{Decay Condition.} Note that $DV = \left(VV'\right) \nu$, so
\begin{align*}
    S(DV, \nu) & = \left(VV'\right)S(\nu,\nu) = \left(VV'\right)\left( \Ric(\nu, \nu) + (n-1) \right)\\
    & = \left(VV'\right)\left( V^2\Ric_{rr} + (n-1) \right)\\
    & = - m(n-1)(n-2)r^{1-n}\cdot(1+m(n-2)r^{-n}).
\end{align*}
As $r\rightarrow\infty$, $S(DV, \nu)\sim - m(n-1)(n-2)r^{1-n}$.
Integrating over the level set $\Sigma_r$ with the volume element $d\sigma_g = r^{n-1} d\sigma_b$ gives
\begin{align*}
    \liminf_{r\to\infty}\int_{\Sigma_r} S(DV, \nu) d\sigma_g & = \lim_{r\to\infty}\int_{\Sigma}\left(- m(n-1)(n-2)r^{1-n}\right)r^{n-1} d\sigma_b\\
    & = - m(n-1)(n-2)\operatorname{Vol}(\Sigma, b) \le 0.
\end{align*}

Finally, if $(\Sigma, b)$ is not Ricci-flat, then $(M, g)$ cannot be of Type II, and hence is not of Obata type.
\end{example}

\begin{rmk}
As noted in \cite{H-J}, in the case $H=-(n-1)$, the integral above coincides with the \emph{Wang–Chruściel–Herzlich mass integral} (see also \cites{C-H, Herzlich, B-C-H}):
$$
-\frac{n-2}2 m(g,V) = \lim_{r\to\infty}\int_{\Sigma_r}(\operatorname{Ric} + (n-1)g)(DV, \nu) d\sigma_g.
$$
Furthermore, this limit is well-defined and converges on any ALH manifold.
\end{rmk}

\bibliographystyle{amsplain}

\end{document}